\documentclass[11pt]{amsart}

\usepackage[marginratio=1:1,height=600pt,width=400pt,tmargin=110pt]{geometry}

\usepackage{amsmath, amsfonts, amssymb, amscd, bezier, amstext, amsthm}
\usepackage{graphicx}
\usepackage{microtype} 
\usepackage{hyperref}
\usepackage{xcolor}

\DeclareMathOperator{\vol}{vol}
\DeclareMathOperator{\spt}{spt}

\DeclareMathOperator{\ind}{ind}

\newcommand{\R}{\mathbb{R}}
\newcommand{\N}{\mathbb{N}}
\newcommand{\Z}{\mathbb{Z}}
\newcommand{\e}{\varepsilon}

\theoremstyle{plain}

\newcounter{alpha}

\newtheorem{prop_}{Proposition}[section]
\newtheorem{prop}[prop_]{Proposition}

\newtheorem{lemma}[prop_]{Lemma}

\newtheorem{definition}[prop_]{Definition}

\newtheorem{cor}[prop_]{Corollary}

\newtheorem{teo}[prop_]{Theorem}

\newtheorem*{obs}{Remark}

\title[\resizebox{5.2in}{!}{The Weyl Law for the phase transition spectrum and density of limit interfaces}]{The Weyl Law for the phase transition spectrum and density of limit interfaces}
\author{Pedro Gaspar} 
\address{Instituto de Matematica Pura e Aplicada (IMPA), Estrada Dona Castorina 
110, 22460-320 Rio de Janeiro, Brazil}
\email{phgms@impa.br}
\author{Marco A. M. Guaraco} 
\address{Department of Mathematics,
The University of Chicago,
5734 S University Ave,
Chicago, IL 60637}
\email{guaraco@math.uchicago.edu}

\thanks{The first author was partly supported by NSF grant DMS-1311795.}

\begin{document}

\begin{abstract} We prove a Weyl Law for the phase transition spectrum based on the techniques of Liokumovich-Marques-Neves. As an application we give phase transition adaptations of the proofs of the density and equidistribution of minimal hypersufaces for generic metrics by Irie-Marques-Neves and Marques-Neves-Song, respectively. We also prove the density of \textit{separating} limit interfaces for generic metrics in dimension 3, based on the recent work of Chodosh-Mantoulidis, and for generic metrics on manifolds containing only separating minimal hypersurfaces, e.g. $H_{n}(M,\Z_2)=0$, for $4\leq n+1\leq 7$. These provide alternative proofs of Yau's conjecture on the existence of infinitely many minimal hypersurfaces for generic metrics on each setting, using the Allen-Cahn approach.
\end{abstract}

\maketitle

\section{Introduction}
In this article we are interested in understanding the limit behavior of solutions to the elliptic \emph{Allen-Cahn equation} on a closed, orientable, Riemannian manifold $M^n$, $n\geq 3$. Namely, we will look at $u:M \to \R$ with
	\begin{equation} \label{eq:ac}
		-\e \Delta u + W'(u)/\e=0, 
	\end{equation}
on the limit when $\e$ goes to zero, where $W$ is a double-well potential, e.g. $W(u)=(1-u^2)^2/4$. This equation and its parabolic counterpart arise in the gradient theory of phase transition phenomena within the van der Waals-Cahn-Hilliard theory \cite{AllenCahn}. Its solutions are known to be related to critical points of the area functional since studied in the context of $\Gamma$-convergence by Modica-Mortola in \cite{ModicaMortola} (see also \cite{KohnSternberg}), where minimizers of the associated energy functional
	\begin{equation}
		E_\e(u) = \int_M \e\frac{|\nabla u|^2}{2} + \frac{W(u)}{\e}, \quad u \in H^1(M),
	\end{equation}
are shown to converge to minimizers of the area functional. The corresponding problem with constraint $\int_M u = c$ is related to constant mean curvature hypersurfaces (we refer the reader to Modica \cite{Modica} and Sternberg \cite{Sternberg}). Since then, strong parallels between these objects have been drawn, see e.g. the surveys \cite{Pacard, Savin} and the references therein.

For more general variational solutions of the Allen-Cahn equation results about the limit behavior were carried out by Hutchinson-Tonegawa \cite{HutchinsonTonegawa}, Tonegawa \cite{Tonegawa}, Tonegawa-Wickramasekera \cite{TonegawaWickramasekera} and the second author \cite{Guaraco}. Roughly speaking, it is known (see \cite{Guaraco} for the precise statement)

\

\noindent\textbf{Convergence Theorem.}\textit{
Let $M^{n+1}$ be a closed Riemannian manifold of dimension $n+1\geq 3$, and let $\{u_{\e_k}\}$ be a sequence of solutions to \eqref{eq:ac} in M with $\e=\e_k \downarrow 0$. Assume that the sequences $\sup_M|u_{\e_k}|$, $E_{\e_k}(u_{\e_k})$ and  $\operatorname{Ind}(u_{\e_k})$ are bounded, where $\operatorname{Ind}(u_\e)$ denotes the Morse index of $u_\e$ as a critical point of $E_\e$. Then as $\e_k \downarrow 0$ its level sets accumulate around a minimal hypersurface $\Gamma \subset M$ which is smooth and embedded outside a singular set of Hausdorff dimension at most $n-7$. Moreover, there are positive integers $m_1,\ldots, m_N$ such that
	\[\lim_k E_{\e_k}(u_{\e_k}) = 2\sigma \sum_{j=1}^N m_j \mathcal{H}^{n}(\Gamma_j),\]
where $\Gamma_1,\ldots, \Gamma_N$ are the connected components of $\Gamma$, and $\sigma = \int_{-1}^1 \sqrt{W(t)/2}\ dt$.
}

\

A minimal hypersurface $\Gamma$, produced in this way is called a \emph{limit interface} and the positive integers $m_j$ are called the \emph{multiplicities} of $\Gamma_j$. 

\begin{obs}
For $n+1=2$ similar conclusions hold, except the regularity of the limit interface. In this case, the limit varifold is supported in an union $\Gamma$ of geodesic arcs with at most $p=\limsup_k \operatorname{Ind}(u_{\e_k})$ junction points, according to \cite{Tonegawa}. It was proved recently by C. Mantoulidis \cite{Mantoulidis1} that if $p=1$ then $\Gamma$ is an immersed geodesic and the possible junction point is a transverse intersection.
\end{obs}

The lower semicontinuity of the index was proven first by Hiesmayr \cite{Hiesmayr}, for two-sided limit interfaces, and then by the first author \cite{Gaspar} in the general case. More precisely, if $\Gamma$ is the limit interface of a sequence of solutions $u_\e$, then $\operatorname{Ind}(\Gamma) \leq  \operatorname{Ind}(u_\e)$, for small $\e$.

Using the Convergence Theorem above, along with min-max techniques for semilinear PDEs, the second author was able to provide an alternative proof of the celebrated result of Almgren-Pitts and Schoen-Simon about the existence of closed minimal hypersurfaces in closed Riemannian manifolds.  This phase transition approach simplifies considerably the variational argument of Almgren-Pitts to prove the existence of a stationary limit but it relies on the regularity theory of Wickramasekera \cite{Wickramasekera}, which is a sharpening of the classical Schoen-Simon compactness theory for stable minimal hypersurfaces. 

In recent work, Chodosh-Mantoulidis \cite{ChodoshMantoulidis} were able to obtain stronger convergence estimates for the case $n+1=3$. Using the work of Ambrosio-Cabr\'e \cite{AmbrosioCabre} and building on the techniques from Wang-Wei \cite{WangWei}, they have shown that, in the situation of the theorem above and in $n+1=3$, level sets of the solutions converge as normal sheets towards a limit interface $\Gamma$. Even more, when $\Gamma$ is a nondegenerate limit interface then all the multiplicities must be one in addition to $\Gamma$ being separating and two-sided. 

Chodosh-Mantoulidis' results have many important consequences. First, they imply a strong version of the \emph{Multiplicty One Conjecture} of Marques and Neves \cite{MarquesNevesIndex} for dimension 3. The  general conjecture claims that for generic metrics in $M^{n+1}$, with $3 \leq n+1 \leq 7$, two-sided unstable components of closed minimal hypersurfaces obtained by min-max methods must have multiplicity one. However, in \cite{ChodoshMantoulidis} it is shown that the two-sided assumption is not necessary when dealing with unstable components of limit-interfaces, regardless of whether they come from a min-max construction or not. Additionally, for generic metrics or for Ricci positive metrics, they show that \textit{all} the components are two-sided and occur with multiplicity one. Second, these results also imply that in dimension 3 it is possible to avoid the use of Wickramasekera's regularity results, making the whole theory considerably more elementary. Finally, in \cite{ChodoshMantoulidis} they prove the upper semicontinuity of the index for any dimension, under the assumption that multiplicities are all one, i.e. if $\Gamma$ is the multiplicity one limit interface of a sequence of solutions $u_\e$, then $\operatorname{Ind}(\Gamma)+\operatorname{Nul}(\Gamma)\leq \operatorname{Ind}(u_\e) +\operatorname{Nul}(u_\e)$, for small $\e$. 

\

In \cite{GasparGuaraco}, the authors generalized the phase transitions approach to min-max theory for minimal hypersurfaces of \cite{Guaraco} to higher dimensional min-max families. More precisely, given $p \in \N$ and small $\e$, there is at least one $u_\e$ which is a solution to (1) with energy $c_\e(p)$ and $\operatorname{Ind}(u_\e) \leq p\leq \operatorname{Ind}(u_\e)+\operatorname{Nul}(u_\e)$. Additionally, the authors proved sublinear bounds for the energy levels of $E_\e$ and small $\e>0$, i.e. if one defines $\ell_p(M)=\lim_{\e \to 0}c_\e(p)$ then there is $C>0$, independent of $p$, such that $C^{-1} p^{\frac{1}{n+1}} \leq \ell_p(M) \leq C p^{\frac{1}{n+1}}$ (see Section \ref{sec:spectrum} for a precise statement). As a consequence, the Convergence Theorem and the upper bound for the index \cite{Gaspar} give the existence of a limit interface of mass $\ell_p(M)$, supported on an minimal hypersurface of index at most $p$, which is smooth and embedded  (away from a small singularity set in dimensions greater than 7). 

The proof of the sublinear bounds for $\ell_p(M)$ is inspired by the techniques used to obtain similar bounds for the spectrum of the volume functional due to Marques-Neves \cite{MarquesNevesInfinite}, which, in turn, are based on the work of Gromov \cite{GromovWaists} and Guth \cite{GuthWidth,Guth}. In light of this analogy, we refer to the sequence $\{\ell_p(M)\}_p$ as \textit{the phase transition spectrum of $M$}.

In \cite{Gromov}, M. Gromov studied several nonlinear analogues of the spectrum of the Laplacian operator in Riemannian manifolds, one of which is precisely the volume spectrum $\{\omega_p(M)\}$. Similarly to the eigenvalues of the Laplacian, the numbers $\omega_p(M)$, called the \emph{$p$-widths} of $M$, are min-max critical values, in this case, for the volume functional. They can be described in the framework of Almgren-Pitts Theory and, in this context, it is shown that $\omega_p(M)$ are achieved by the volume of minimal hypersurfaces, possibly with a small singular set, and index at most $p$. We refer to \cite{Guth,MarquesNevesInfinite,MarquesNevesIndex} for the precise definitions and statements of the theorems regarding the $p$-widths. Motivated by Weyl's asymptotic law for the eigenvalues of the Laplacian, Gromov conjectured \cite{GromovSingularities, GromovWaists} that the volume spectrum should satisfy a similar Weyl law. This result was recently confirmed by Liokumovich-Marques-Neves in \cite{LioMarNev} and it has been since used to derive deep geometric consequences about minimal hypersurfaces \cite{IrieMarNev,MarNevSon}, most notably the confirmation for generic metrics of Yau's conjecture, concerning the existence of infinitely many minimal hypersurfaces.

\

In this work we show that a Weyl Law also holds for $\{\ell_p(M)\}_p$. Our main result is:

\

\noindent\textbf{Main Theorem.} (Weyl Law for the phase transition spectrum)\textit{  There exists $\tau(n)>0$ such that
	\[\lim_{p \to +\infty} p^{-\frac{1}{n+1}} \ell_p(M) = \tau(n)\vol(M,g)^{\frac{n}{n+1}},\]
for all compact Riemannian manifolds $(M^{n+1},g)$ possibly with a nonempty piecewise smooth boundary.}

\

One is led then to the problem of describing the limit interfaces which arise from the Allen-Cahn strategy. For results along these lines see, for instance, \cite{Mantoulidis1,WangWei}, for finite index solutions on surfaces, \cite{GasparGuaraco} for least area limit interfaces (in the sense of Mazet-Rosenberg \cite{MazetRosenberg}) and \cite{ChodoshMantoulidis} for $n+1=3$.

\

A \textit{generic} set on a Banach space is one that is the intersection of countably many open and dense sets. By the Baire Category Theorem such a set must be dense. From the work of B. White \cite{WhiteBumpy} it is known that there exists set of metrics on $M$, generic in the $C^\infty$ topology, for which all minimal hypersurfaces are nondegenerate. Combining this fact, with the finite index min-max constructions by the authors in \cite{GasparGuaraco}, the multiplicity one and upper semicontinuity results from \cite{ChodoshMantoulidis}, one obtains an alternative proof of Yau's conjecture on the existence of infinitely many minimal surfaces in 3-dimensional manifolds. For dimensions $3\leq n+1\leq 7$ and generic metrics, the conjecture follows from the fact that the union of all closed, smooth and embedded minimal hypersurfaces is dense, as it was proved by Irie-Marques-Neves \cite{IrieMarNev} using the Morse theoretic techniques developed by Marques-Neves, together with several other authors, in recent years \cite{LioMarNev,MarquesNevesIndex,MarquesNevesInfinite}. The main tool is the Weyl Law for the spectrum of the volume functional proved by Liokumovich-Marques-Neves \cite{LioMarNev}, which describes the asymptotic behavior of the areas of the min-max minimal hypersurfaces constructed by Marques-Neves \cite{MarquesNevesInfinite}.

As an application of our main theorem, we also obtain an alternative proof of the density and equidistribution of min-max minimal hypersurfaces of \cite{IrieMarNev}, which also implies Yau's conjecture for dimensions $3\leq n+1\leq 7$. Stretching the techniques further and using the constructions of Pacard-Ritor\'e \cite{PacardRitore} and the convergence results from Chodosh-Mantoulidis \cite{ChodoshMantoulidis}, as well as a local version of the Bumpy Metrics Theorem of B. White \cite{White1}, we also obtain the density of separating limit interfaces for generic metrics (or for Ricci positive metrics) in 3-dimensional manifolds. The Local Bumpy Metric Theorem was announced by B. White in \cite{WhiteBumpy}. The proof follow the arguments in \cite{White1}. Here we have decided to include the setting for the local version we need, i.e. case of hypersurfaces, as well as the necessary modifications for its proof. 

Similar ideas show that limit interfaces are dense for generic metrics in manifolds that contain no non-separating minimal hypersurfaces, e.g. $H_{n}(M,\Z_2)=0$.

\subsection*{Outline of the paper.} In Section \ref{sec:spectrum} we define the phase transition spectrum of an open set in a Riemannian manifold, which is obtained as the volume of limit interfaces of min-max solutions to equation (1). In Section \ref{sec:weyl} we prove the Weyl Law for open sets of the Euclidean space. In Section \ref{sec:weylman} we prove the Weyl Law for closed Riemannian manifolds. Finally, in Section \ref{sec:density} we apply the result to the study of the density and equidistribution of minimal hypersurfaces and the density of limit interfaces and the local version of the bumpy metrics theorem.

\subsection*{Acknowledgements.} Both authors would like to thank Fernando C. Marques and Andr\'e Neves for useful discussions and their interest in this work. The first author is grateful to the Department of Mathematics at Princeton University for its hospitality. Part of this work and the first drafts were carried out while visiting during the academic year of 2017-18. The second author would like to thank FIM - ETH, Zurich for their kind hospitality, where this work was finished during a visit in Spring 2018.

\section{The phase transition spectrum} \label{sec:spectrum}
In this section we recall the definition of the min-max values $\{c_\e(p)=c_\e(p,M)\}_{p \in \N}$ for the energy functional and we define the phase transition spectrum, following the construction of \cite{GasparGuaraco}. Let $(M^{n+1},g)$, be a compact Riemannian manifold of dimension $n+1 \geq 3$ with a possibly nonempty piecewise smooth boundary. We will assume hereafter that $W \in C^3(\R)$ is an \emph{even} nonnegative function which satisfies the following condition:

\begin{itemize}
	\item[(A)] $W$ is a \emph{double-well potential}, namely it has exactly three critical points, two of which are non-degenerate minima at $\pm 1$ with $W(\pm 1)=0$ and $W''(\pm 1)>0$, and the third is a local maximum point at the origin.
\end{itemize}

We intend to study the critical points of the Allen-Cahn \emph{energy functional}
	\begin{align*}
		E_\e(u) = \int_M \frac{\e|\nabla u|^2}{2} + \frac{W(u)}{\e}, \quad u \in H^1(M),
	\end{align*}
which are precisely the solutions to the equation \eqref{eq:ac}, and their properties as the parameter $\e$ converges to $0$. Since $E_\e$ is an even functional, we can use families of symmetric, compact and topologically non-trivial subsets of $H^1(M)$ to detect critical points of this functional. This approach was adopted in \cite{GasparGuaraco} and it is inspired by the multiparameter min-max construction for the area functional of \cite{MarquesNevesInfinite} (see also \cite{Guth,Gromov}).

\subsection{A topological \texorpdfstring{$\Z/2\Z$}{Z/2Z} index} Recall that to each (para)compact symmetric $A \subset H^1(M)$ we can associate a nonnegative integer which we will call its $\Z/2\Z$ (or \emph{cohomological}) \emph{index} in the following manner. If $0 \notin A$ one verifies that there exists an odd continuous map $f:A \to S^N$ for some $N \in \N\cup\{\infty\}$. Here we consider $S^N$, $N \in \N$, as finite dimensional spheres in $S^\infty = \bigcup_k S^k$ with the topology given by the direct limit of $\{S^k\}_{k \in \N}$ ordered by the inclusions $S^k \to S^{k'}$ for $k \leq k'$. Denote by $\tilde f: \tilde A \to \R\mathbb{P}^\infty$ the induced continuous map, where $\tilde A$ and $\R\mathbb{P}^\infty$ are the orbit spaces $A/\{x \sim -x\}$ and $S^\infty/\{x\sim -x\}$ respectively. Recall the cohomology ring of the infinite dimensional projective space $\R\mathbb{P}^{\infty}$ with $\Z/2\Z$ coefficients is isomorphic to the polynomial ring $(\Z/2\Z)[w]$, with a generator $w \in H^1(\R\mathbb{P}^\infty, \Z/2\Z)$. Moreover, the map $\tilde f$ is unique modulo homotopy, thus we may define the cohomological index of $A$ by
	\[\ind(A) := \sup\{k : \tilde f^*(w^{k-1}) \neq 0 \in H^{k-1}(A,\Z/2\Z)\}.\]
We set $w^0=1\in H^0(\R\mathbb{P}^\infty,\Z/2\Z)$ and adopt the convention $\ind(\emptyset)=0$. Moreover we let $\ind(A)=+\infty$ whenever $0 \in A$. This index was studied was studied by Fadell and Rabinowitz in \cite{FadellRabinowitz1}, in the context of bifurcation theory, among others, and it measures the cohomological non-triviality of $A$. The set of all paracompact symmetric subsets of $H^{1}(M)$ will be denoted by $\mathcal{C}$. We list here some of the properties of $\ind$.

\begin{enumerate}
	\item[(I1)](Normalization) $\ind(A)=0$ if, and only if, $A=\emptyset$.
	\item[(I2)](Monotonicity) If $A_1,A_2 \in \mathcal{C}$ and there exists an equivariant continuous map $A_1 \to A_2$, then $\ind(A_1) \leq \ind(A_2)$.
	\item[(I3)](Continuity) If  $X \subset A$ is an invariant closed subset of $A$, there exists an invariant neighborhood $V \subset A$ of $X$ such that $\ind(X) = \ind(\overline V)$.
	\item[(I4)](Subadditivity) For all paracompact symmetric $A_1,A_2 \subset H^1(M)$ we have $\ind(A_1 \cup A_2) \leq \ind(A_1) + \ind(A_2)$.
	\item[(I5)] For every paracompact symmetric $A \subset H^1(M)$, if $\ind(A)\geq 1$, then the orbit space $\tilde A$ has infinitely many elements.
	\item[(I6)] It holds $\ind(A)<+\infty$ for all \emph{compact} $A\subset H^1(M)\setminus\{0\}$. More generally $\ind(A) \leq \dim A$ where $\dim A$ is the covering dimension of $A$.
\end{enumerate}

More details on the construction of this invariant may be found also in the Appendix of \cite{GasparGuaraco}. 

\subsection{Min-max construction for the energy functional} We can use the index $\ind$ to develop a $\Z/2\Z$-equivariant min-max construction for $E_\e$ following the general setting of \cite{Ghoussoub}. For each $p \in \N$ let
	\[\mathcal{F}_p(M) = \mathcal{F}_p := \{ A \subset H^1(M) : A \ \mbox{compact, symmetric,} \ \ind(A) \geq p+1  \}.\]
One verifies that $\mathcal{F}_p$ is a $p$-dimensional $\Z/2\Z$-cohomological family in the sense of \cite{Ghoussoub}, so we may expect that the associated \emph{min-max values}
	\[ c_\e(p,M):=\inf_{A \in \mathcal{F}_p(M)} \sup_{u \in A} E_\e(u), \quad p \in \N, \]
are achieved by critical points of $E_\e$. Here some remarks are in order. Firstly we are interested in \emph{non-constant} critical points, so we may expect $c_\e(p,M)>0$, what can be proved using the same strategy of \cite[\S 4]{Guaraco}, and that $E_\e(u)=c_\e(p,M)$ for a nonzero critical point. In this regard we also note that the results of \cite{Ghoussoub} cannot be directly applied to $E_\e:H^1(M) \to \R$, since the $\Z/2\Z$ action $x \mapsto -x$ fixes the origin. In light of these observations, we have the following existence theorem.

We denote by $K_c$, for $c \in \R$, the set of critical points for $E_\e$ with energy $c$. Moreover given $m \in \N$ we let $K_c(m)=\{u \in K_c : m(u)\leq m\}$, where $m(u)$ denotes the Morse index of $u$.

\begin{teo}[\cite{GasparGuaraco}] \label{teo:exist} Fix $\e>0$.
	\begin{enumerate}
	\item[(1)] For every $p \in \N$ it holds $0 < c_\e(p,M) \leq E_\e(0) = \vol(M,g)W(0)/\e$.
	\item[(2)] If $c_\e(p,M)<E_\e(0)$ then there is a critical point $u \in K_{c_\e(p,M)}(p)$ such that $|u|\leq 1$. Moreover if $c_\e(p,M) = c_\e(p+k,M)$ for some $k \in \N$ then
		\[\ind(K_{c_\e(p,M)}(p+k)) \geq k+1.\]
	\item[(3)] There is $p_0=p_0(\e,M) \in \N$ such that $c_\e(p,M)=E_\e(0)$ for all $p \geq p_0$.
\end{enumerate}
\end{teo}

\begin{obs}
In \cite{GasparGuaraco}, this theorem was proved for compact manifolds with empty boundary. The adaptations to the case $\partial M \neq \emptyset$ and the boundary is Lipschitz are straightforward. In this case the solutions $u \in K_{c_\e(p,M)}$ are weak solutions of the Neumann problem associated to the \eqref{eq:ac} with $|u|\leq 1$. By elliptic regularity we see that $u$ is of class $C^3$ in the interior of $\Omega$. Moreover if $\partial M$ is sufficiently regular, we get also $u \in C^3(\overline \Omega)$. The existence question for the Neumann problem in Euclidean domains was previously tackled by G. Vannella in \cite{Vannella}.
\end{obs}

By the Convergence Theorem one expects to obtain minimal hypersurfaces from $u_\e \in K_{c_\e(p,M)}$ by making $\e \downarrow 0$, for a fixed $p \in \N$. For this purpose we need to check that $c_\e(p)$ stays bounded away from both $0$ and $+\infty$ for small $\e>0$. This is the content of the following theorem proved in \cite{GasparGuaraco} and inspired by the works of Gromov \cite{Gromov}, Guth \cite{GuthWidth, Guth} and Marques-Neves \cite{MarquesNevesInfinite}. We note that, in addition to providing these uniform bounds for $c_\e(p)$, the theorem describes how $\lim_{\e}c_{\e}(p,M)$ grows with respect to $p$.

\begin{teo}[\cite{GasparGuaraco}] \label{teo:bounds}
There exists a constant $C=C(M)>1$ such that 
	\[ C^{-1}p^{\frac{1}{n+1}} \leq \liminf_{\e \to 0^+} c_\e(p,M) \leq \limsup_{\e \to 0^+}c_\e(p,M) \leq Cp^{\frac{1}{n+1}}.\]
\end{teo}

\begin{obs} Once again these bounds are proved in \cite{GasparGuaraco} for closed manifolds. The proof of the upper bounds for the case of nonempty piecewise smooth $\partial M$ are the same, given the triangulability of such manifolds. The proof of lower bounds are also similar. In fact, the energy density estimate of \cite[Lemma 5.2]{GasparGuaraco} still holds for small balls contained in $M\setminus \partial M$. Moreover for each integer $p$ we can pick $p$ disjoint balls in $M\setminus \partial M$ of radius $r_p=\nu p^{-1/(n+1)}$ for a small $\nu>0$ depending on $M$. Since the variational construction for the energy functional works verbatim the rest of the proof can be carried out in the context of nonempty boundary.
\end{obs}

As an interesting consequence of Theorems \ref{teo:exist} and \ref{teo:bounds} we remark that the number of solutions of \eqref{eq:ac} grows to $+\infty$ as $\e$ goes to $0$, since $E_\e(0)\uparrow +\infty$ as $\e \downarrow 0$. Regarding minimal hypersurfaces in closed manifolds, the Convergence Theorem (see \cite{Guaraco}) and the index bounds of \cite{Gaspar} imply:

\begin{cor}[\cite{Gaspar,GasparGuaraco}] \label{cor:min}
Assume $M^{n+1}$ is closed and that $n+1 \geq 3$, fix $p \in \N$ and choose $\{\e_k\}_{k \in \N}$ such that $c_{\e_k}(p,M) \to c(p,M)$. There exists an integral varifold $V_p$ such that
	\begin{enumerate}
		\item[(i)] $||V_p||(M) = c(p,M)/2\sigma$.
		\item[(ii)] $V_p$ is stationary in $M$.
		\item[(iii)] $\mathrm{sing} V$ has Hausdorff dimension $\leq n-7$.
		\item[(iv)] $\mathrm{reg} V$ is an embedded minimal hypersurface of Morse index $\leq p$.
	\end{enumerate}
\end{cor}

As we mentioned, when $\partial M$ is nonempty the solutions provided by Theorem \ref{teo:exist} are weak solutions for the Neumann problem associated to the Allen-Cahn equation. In view of the varifold convergence it is reasonable to expect that solutions with bounded energy, $L^\infty$-norm and Morse index give rise to free boundary minimal hypersurfaces in $M$. The main issue with this expectation is the lack of a boundary regularity theorem in the spirit of \cite{Wickramasekera} and \cite{TonegawaWickramasekera} (compare with the closed 3-dimensional case in \cite{ChodoshMantoulidis}). Regarding solutions which are local minimizers of the volume-constrained energy functional, a similar statement is proved in \cite{HutchinsonTonegawa}, see Theorem 3.

\subsection{The phase transition spectrum} We will denote hereafter
	\[ \gamma(M) := \inf\left\{ \gamma>0 : p^{-\frac{1}{n+1}}\limsup_{\e \to 0^+} c_{\e}(p,M) \leq \gamma \ \mbox{for all} \ p \right\}.\]
It follows from Theorem \ref{teo:bounds} and the Remarks from last section that $\gamma(M)$ is a positive real number. We define the \emph{upper} and \emph{lower phase transition spectra} of $(M,g)$ as the sequences $\{\overline \ell_p(M)\}_{p \in \N}$ and $\underline \ell_p(M)\}_{p \in \N}$, respectively, given by
	\[\overline \ell_p(M) = \frac{1}{2\sigma}\limsup_{\e \to 0^+} c_\e(p,M)\]
and
	\[\underline \ell_p(M) = \frac{1}{2\sigma}\liminf_{\e \to 0^+} c_\e(p,M).\]
These sequences may be seen as analogues, in the context of phase transition, of the $p$-widths $\omega_p(M)$ -- that is, the \emph{volume spectrum} -- for the area functional, as defined in \cite{Gromov,MarquesNevesInfinite}. We also remark that, by \cite{GasparGuaraco}, the following comparison between phase transition and volume spectra holds
	\[\omega_p(M) \leq \underline\ell_p(M) \leq \overline \ell_p(M), \quad \mbox{for all} \quad p \in \N.\]

\begin{obs}
One can use Corollary \ref{cor:min} to show that if $M^{n+1}$ is a closed manifold of dimension $3 \leq n+1 \leq 7$ then the upper and lower phase transition spectra coincide. In fact if $\underline \ell_p(M) < \overline \ell_p(M)$ for some $p \in \N$ then we can construct, for each $s \in (\underline \ell_p(M), \overline \ell_p(M))$ a sequence $\{\e_k\}$ with $\e_k \downarrow 0$ such that the varifolds associated to a sequence of solutions $\{u_k\}$ with $E_{\e_k}(u_k) = c_{\e_k}(p)$ and index $\leq p$ converge to an integral stationary varifold $V(s)$ with $||V(s)||=s$ and such that $\spt V(s)$ is a smooth and embedded minimal hypersurface of index $\leq p$. In particular, the number of minimal hypersufaces with area $\leq \overline \ell_p(M)$ and index $\leq p$ cannot be finite. By the Compactness theorem of B. Sharp \cite{Sharp} and the Bumpy Metrics Theorem of B. White \cite{WhiteBumpy} we see that this can only happen for a meagre set of metrics in $M$. On the other hand the limit spectrum values $\overline \ell_p(M)$ and $\underline \ell_p(M)$ depend continuously on the metric in $M$, see Lemma \ref{lem:cont} below. Hence $\overline\ell_p(M)=\underline\ell_p(M)$ for all $p$ and all metrics on $M$.
\end{obs}

\begin{definition}From now on we will denote by $\{\ell_p(M)\}$ the sequence $\{\underline\ell_p(M)\}$,  which coincide with  $\{\bar\ell_p(M)\}$ for $3\leq n+1\leq 7$, in view of the Remark above. \end{definition}

We state next the Weyl Law for the phase transition spectrum, which is our main result.

\begin{teo} \label{thm:weyl}
There exists a universal $\tau(n)>0$, such that
\[\lim_{p \to +\infty} p^{-\frac{1}{n+1}} \ell_p(M) = \tau(n)\vol(M,g)^{\frac{n}{n+1}}.\]
\end{teo}

\section{Weyl law for Euclidean domains}\label{sec:weyl}
In this section we prove the Weyl law for the phase transition spectrum $\{\ell_p(\Omega)\}$ for Euclidean bounded domains $\Omega \subset \R^{n+1}$ with piecewise smooth boundary, for $3\leq n+1\leq 7$, following the strategy of Liokumovich-Marques-Neves \cite{LioMarNev}. We denote by $C$ the unit cube in $\R^{n+1}$, and we say that two regions $\Omega_1,\Omega_2 \subset \R^{n+1}$ are \emph{similar} if they differ by isometries and scaling. In this case, if $\Omega_2=T(\Omega_1)$ where $T$ is a composition of such maps with scaling factor $\lambda>0$, then for all $u:\Omega_2 \to \R$ and $\e>0$ it holds
	\begin{align}
		E_\e(u \circ T) &= \int_{\Omega_1} \left( \frac{\e|\nabla (u \circ T)|^2}{2} + \frac{W(u\circ T)}{\e}\right)\,d\mathcal{L}^{n+1} \nonumber\\
		& = \int_{\Omega_2} \left( \frac{\e\,\lambda^2 |\nabla u|^2}{2} + \frac{W(u)}{\e}\right)\lambda^{-(n+1)}\,d\mathcal{L}^{n+1}\\
		& = \lambda^{-n} E_{\lambda \e}(u). \nonumber
	\end{align}
Consequently we have
	\begin{equation} \label{eq:scal}
		c_{\lambda\e}(p, \lambda\Omega_1) = \lambda^{n} c_{\e}(p,\Omega_1)
	\end{equation}
for all $p \in \N$. In particular the energy spectrum scales as the $n$-th power of the scaling factor. The following lemma generalizes this transformation property and it will be useful throughout the article.

\begin{lemma} \label{lema:var}
Let $F:(\Omega_1,g_1) \to (\Omega_2,g_2)$ be a diffeomorphism between compact $(n+1)$-manifolds with piecewise $C^1$ boundary. For all $\e>0$ and $u \in H^1(\Omega_2)$ it holds
	\[E_{\e/||DF||_{\infty}}(u \circ F, \Omega_1) \leq ||DF||_{\infty}||DF^{-1}||_{\infty}^{n+1} E_\e(u,\Omega_2).\]
\end{lemma}

\begin{proof}
Let $\lambda=||DF||_{\infty}$. Since
	\[g_1(\nabla^{g_1} (u\circ F),v) = d(u\circ F)(v) = g_2(\nabla^{g_2} u, dF(v))\]
we obtain
	\[|\nabla^{g_1}(u \circ F)|^2_{g_1} \leq \lambda^2|\nabla^{g_2} u|^2_{g_2} \circ F.\]
Moreover by Hadamard's inequality we can bound the norm of the Jacobian determinant of $F^{-1}$ from above by $||DF^{-1}||_{\infty}^{n+1}$. Thus using the change of variables formula we obtain
	\begin{align*}
		E_{\e/\lambda}(u\circ F,\Omega_1) & = \int_{\Omega_1} \left( \frac{\e|\nabla^{g_1}(u\circ F)|_{g_1}^2}{2\lambda} + \frac{\lambda W(u\circ F)}{\e} \right)\, d\mathcal{H}^{n+1} \\
		& \leq \lambda \int_{\Omega_1} \left(\frac{\e(|\nabla^{g_2} u|_{g_2}^2 \circ F)}{2} + \frac{ W(u \circ F)}{\e}\right)|JF^{-1}||JF|\,d\mathcal{H}^{n+1}\\
		&\leq \lambda||DF^{-1}||^{n+1} \int_{\Omega_2} \left(\frac{\e|\nabla^{g_2} u|_{g_2}^2}{2} + \frac{W(u)}{\e} \right) \, d\mathcal{H}^{n+1} \qedhere
	\end{align*}
\end{proof}

The next Lemma is the phase transition version of the important Lusternik-Schnirelmann inequality of \cite{LioMarNev}.
	
\begin{lemma}[Lusternik-Schnirelman Inequality]
Consider domains $\Omega,\{\Omega_i\}_{i=1}^N$ and $\{\Omega_i^*\}_{i=1}^N$ in $\R^{n+1}$ with piecewise smooth boundaries such that
	\begin{itemize}
		\item $|\Omega|=|\Omega_i|=1$ for $i=1,\ldots, N$
		\item $\Omega_i^*$ is similar to $\Omega_i$ for $i=1,\ldots, N$
		\item $\{\Omega_i^*\}$ are pairwise disjoint subsets of $\Omega$.
	\end{itemize}
Then
	\[p^{-\frac{1}{n+1}}\ell_p(\Omega) \geq \sum_{i=1}^N |\Omega_i^*| p_i^{-\frac{1}{n+1}}\ell_{p_i}(\Omega_i) - \frac{c}{pV},\]
where $p_i=\lfloor p|\Omega_i^*| \rfloor$, $V = \min_i\{|\Omega_i^*|\}$ and $c=\max_i \gamma(\Omega_i)$.
\end{lemma}

\begin{proof}
Denote
	\[\bar p := \sum_{i=1}^N p_i \leq \sum_{i=1}^N p|\Omega_i^*| \leq p|\Omega|=p.\]
Given $\e>0$, $A \in \mathcal{F}_p(\Omega)$ and $\delta>0$ for each $i=1,\ldots, N$ consider
	\[ A_i = \{u \in A : E_\e(u,\Omega_i^*) \leq c_\e(p_i, \Omega_i^*) -\delta/N\}. \]
By definition we have $\ind(A_i) \leq p_i$. Hence by the subadditivity of $\ind$ we see that there exists $u \in A \setminus \cup_{i=1}^N A_i$ and
	\[\sup_A E_\e \geq E_\e(u, \Omega) \geq \sum_{i=1}^N E_\e(u, \Omega_i^*) > \sum_{i=1}^N c_\e(p_i,\Omega_i^*) - \delta.\]
Therefore
	\begin{equation}
		c_\e(p,\Omega) \geq \sum_{i=1}^N c_\e(p_i,\Omega_i^*).
	\end{equation}
Since $\Omega_i$ is isometric to $|\Omega_i^*|^{-1/(n+1)}\Omega_i^*$ and
	\[1 \geq \frac{p_i}{p|\Omega_i^*|} \geq 1- \frac{1}{p|\Omega_i|} \geq \left( 1-\frac{1}{p|\Omega_i|} \right)^{n+1} ,\]
we conclude that
	\begin{align*}
		p^{-\frac{1}{n+1}}c_\e(p,\Omega) & \geq p^{-\frac{1}{n+1}} \sum_{i=1}^N c_\e(p_i,\Omega_i^*)\\
		& = \sum_{i=1}^N |\Omega_i^*| \left( \frac{p_i}{p|\Omega_i^*|} \right)^{\frac{1}{n+1}} p_i^{-\frac{1}{n+1}} c_{\e/|\Omega_i^*|}(p_i,\Omega_i)\\
		& \geq \sum_{i=1}^N |\Omega_i^*| \left( 1 - \frac{1}{p|\Omega_i|} \right) p_i^{-\frac{1}{n+1}} c_{\e/|\Omega_i^*|}(p_i,\Omega_i).
	\end{align*}
By making $\e \downarrow 0$ we get
	\begin{align*}
		p^{-\frac{1}{n+1}}\ell_p(\Omega) &\geq \sum_{i=1}^N |\Omega_i^*| p_i^{-\frac{1}{n+1}} \ell_{p_i}(\Omega_i) - \frac{1}{p} \sum_{i=1}^N\frac{|\Omega_i^*|}{\min_j |\Omega_j^*|} p_i^{-\frac{1}{n+1}} \ell_{p_i}(\Omega_i)\\
		& \geq \sum_{i=1}^N|\Omega_i^*|p_i^{-\frac{1}{n+1}}\ell_{p_i}(\Omega_i) - \frac{\max_j \gamma(\Omega_j)|\Omega|}{pV}.\qedhere
	\end{align*}
\end{proof}

\begin{teo} \label{thm:domain}
There is a positive constant $\tau(n,W)>0$ such that for all Lipschitz domains $\Omega \subset \R^{n+1}$ with piecewise smooth boundary it holds
	\[\lim_{p \to +\infty} p^{-\frac{1}{n+1}} \ell_p(\Omega) =\tau(n,W)|\Omega|^{\frac{n}{n+1}}.\]
\end{teo}

It follows from \eqref{eq:scal} that we may assume, without loss of generality, that $|\Omega|=1$. As in \cite{LioMarNev}, we write $\tilde \ell_p = p^{-1/(n+1)}\ell_p$. We first prove the following

\begin{lemma}
$\displaystyle \liminf_{p } \tilde \ell_p(C) = \limsup_{p} \tilde\ell_p(C)$
\end{lemma}

\begin{proof}
Choose sequences $\{p_k\}_{k \in \N}$ and $\{q_j\}_{j \in \N}$ such that
	\[\lim_k \tilde\ell_{p_k}(C) = \limsup_p \tilde\ell_p(C), \quad \mbox{and} \quad \lim_j \tilde\ell_{q_j}(C) = \liminf_p \tilde\ell_p(C).\]
For a fixed $k$ and all $j$ such that $\delta_j:=p_k/q_j<1$ let $N_j$ be the maximal number of cubes $\{C_i^*\}_{i=1}^{N_j}$ of volume $\delta_j$ contained in $C$ and with pairwise disjoint interiors. Note that $\delta_j N_j =\sum_i |C_i^*| \to 1$ as $j\to \infty$. Since
	\[\lfloor q_j|C_i^*| \rfloor = \lfloor q_j\delta_j \rfloor = p_k\]
from the Lusternik-Schnirelman inequality we obtain
	\[ \tilde \ell_{q_j}(C) \geq \sum_{i=1}^{N_j} |C_i^*| \tilde\ell_{p_k}(C) - \frac{\gamma(C)}{q_j\delta_j} = N_j\delta_j \tilde\ell_{p_k}(C) - \frac{\gamma(C)}{p_k}.\]
By letting $j \to + \infty$ it follows that
	\[ \liminf_p \tilde \ell_p(C) \geq \tilde\ell_{p_k}(C) - \frac{\gamma(C)}{p_k} \]
and thus
	\[ \liminf_p \tilde\ell_p(C) \geq \limsup_p\tilde\ell_p(C). \]
\end{proof}

We denote $\tau(n,W) := \lim_p\tilde\ell_p(C)$. Next, we prove that $\tilde\ell_p(\Omega) \to \tau(n,W)$. 

\begin{lemma} \label{lema:below}
$\displaystyle \liminf_{p} \tilde \ell_p(\Omega) \geq \tau(n,W)$.
\end{lemma}

\begin{proof}
Given $\delta>0$ there is a family $\{C_i^*\}_{i=1}^N$ of cubes with volume $v_i \in (0,1)$ contained in $\Omega$ with pairwise disjoint interiors and
	\[\sum_{i=1}^N |C_i^*| \geq 1-\delta.\]
From Lusternik-Schnirelman inequality we get
	\[\tilde\ell_p(\Omega) \geq \sum_{i=1}^N |C_i^*| \tilde\ell_{\lfloor p v_i \rfloor}(C) - \frac{\gamma(C)}{p\min_i v_i}.\]
Hence
	\[\liminf_p \tilde\ell_p(\Omega) \geq (1-\delta) \liminf_p \tilde\ell_p(C) = (1-\delta)\tau(n,W).\]
Since $\delta>0$ is arbitrary, this concludes the proof.
\end{proof}

The following lemma is proved in \cite{LioMarNev} and it shows that we can fill as much of the volume of $C$ as we want by domains which are similar to $\Omega$.

\begin{lemma} \label{lema:tess}
There is a sequence $\{\Omega_i^*\}_{i \in \N}$ of domains contained in $C$ which are similar to $\Omega$ and have pairwise disjoint interiors such that for all $\delta>0$ we can find $N=N(\delta) \in \N$ satisfying
	\[ \sum_{i=1}^N |\Omega_i^*| \geq 1-\delta.\]
\end{lemma}

We conclude the proof of Theorem \ref{thm:domain} by proving:

\begin{lemma}
$\displaystyle \tau(n,W) \geq \limsup_p \tilde\ell_p(\Omega).$
\end{lemma}

\begin{proof}
As in \cite{LioMarNev}, we will use the previous lemma for each cube of a maximal disjoint collection of cubes in $C$ with small volume. By applying the Lusternik-Schnirelman inequality to $C$ and the family of all these domains similar to $\Omega$ we get the desired inequality.

Choose a sequence $\{q_k\}_k$ such that $\lim_k \tilde\ell_{q_k}(\Omega) = \limsup_p\tilde\ell_p(\Omega)=:\beta$. Consider the family $\{\Omega_i^*\}$ given by Lemma \ref{lema:tess}. For a fixed $k$ and all $p \in \N$ such that $\delta_p := q_k/(p|\Omega_1^*|)<1$ let $N_p$ be the maximal number of cubes $\{C_j^*\}_{j=1}^{N_p}$ contained in $C$ with pairwise disjoint interiors and volume $\delta_p$. Again we have $\delta_pN_p \to 1$ as $p\to+\infty$.

For all $\delta>0$ and each $j=1,\ldots,N_p$, by Lemma \ref{lema:tess} we can choose regions $\{\Omega_{i,j}\}_{i=1}^N$ inside $C_j^*$ with pairwise disjoint interiors and similar to $\Omega$ such that
	\[|\Omega_{i,j}|=|C_j^*||\Omega_i^*| =\delta_p |\Omega_i^*|.\]
If $v=\min\{|\Omega_i^*|: i=1,\ldots, N\}$ and
	\[p_i:=\lfloor p|\Omega_{i,j}| \rfloor = \lfloor p\delta_p |\Omega_i^*| \rfloor = \left\lfloor q_k \frac{|\Omega_i^*|}{|\Omega_1^*|} \right\rfloor,\]
then
	\[\min\{|\Omega_{i,j}|: i=1,\ldots, N, j=1,\ldots, N_p\} = \delta_p v\]
and
	\begin{align*}
		\tilde\ell_p(C) & \geq \sum_{j=1}^{N_p} \sum_{i=1}^N |\Omega_{i,j}|\tilde\ell_{p_i}(\Omega) - \frac{\gamma(\Omega)}{p\delta_p v} \\
		&= \delta_p N_p\left( |\Omega_1^*|\tilde\ell_{q_k}(\Omega) + \sum_{i=2}^N|\Omega_i^*|\tilde\ell_{p_i}(\Omega) \right) - \frac{\gamma(\Omega)|\Omega_1^*|}{q_k v}.
	\end{align*}
Hence
	\[\tau(n,W) \geq |\Omega_1^*|\tilde\ell_{q_k}(\Omega) + \sum_{i=2}^N |\Omega_i^*|\tilde\ell_{\left\lfloor q_k |\Omega_i^*|/|\Omega_1^*| \right\rfloor}(\Omega) - \frac{\gamma(\Omega)|\Omega_1^*|}{q_kv}\]
and by letting $k\to +\infty$ and using Lemma \ref{lema:below} we get
	\[\tau(n,W) \geq |\Omega_1^*|\beta + \liminf_p \tilde\ell_p(\Omega) \sum_{i=2}^N |\Omega_i^*| \geq |\Omega_1^*|\beta + \tau(n,W)(1-\delta-|\Omega_1^*|).\]
Therefore
	\[(\delta + |\Omega_1^*|)\tau(n,W) \geq |\Omega_1^*|\beta\]
and, since $\delta$ is arbitrary, $\tau(n,W) \geq \beta$, as we wanted to prove.
\end{proof}

\section{Weyl Law for closed manifolds}\label{sec:weylman}
Consider a closed Riemannian manifold $(M^{n+1},g)$. As in the previous section, we will denote by $C$ the unit cube in $\R^{n+1}$ and 
	\[\tau(n) = \tau(n,W) = \lim_p \tilde\ell_p(C) = \lim_p p^{-\frac{1}{n+1}}\ell_p(C).\]
In this section we prove:

\begin{teo}[Weyl Law for the phase transition spectrum] \label{thm:weyl_}
For all compact Riemannian manifolds $(M^{n+1},g)$, possibly with a nonempty piecewise smooth boundary, it holds
	\[\lim_{p \to +\infty} p^{-\frac{1}{n+1}} \ell_p(M) = \tau(n)\vol(M,g)^{\frac{n}{n+1}},\]

\end{teo}

First, we show:

\begin{prop}
It holds
	\[\liminf_{p \to + \infty} p^{-\frac{1}{n+1}}\ell_p(M) \geq \tau(n) \vol(M,g)^{\frac{n}{n+1}}.\]
\end{prop}

\begin{proof}
Without loss of generality we may assume $\vol(M,g)=1$. Given $\delta>0$ there is $\bar r>0$ such that for all $r \in (0,\bar r]$ and $ x \in M$ the Euclidean metric $g_0=(\exp_x^{-1})^*g_{eucl}$ induced on the geodesic ball $\mathcal{B}_r(x) \subset M$ satisfies $(1+\delta)^{-2}g \leq g_0 \leq (1+\delta)^{2}g$, and consequently
	\[(1+\delta)^{-(n+1)} \vol(\mathcal{B}_r(x)) \leq |B_r(0)| \leq (1+\delta)^{n+1} \vol(\mathcal{B}_r(x)).\]
Moreover, by arguing as we did in Lemma \ref{lema:var} we see that for all $\e>0$ and $u \in H^1(\mathcal{B}_r(x))$ we have
	\[E_\e(u,\mathcal{B}_r(x)) \geq (1+\delta)^{-(n+2)} E_{\e/(1+\delta)}(u \circ \exp_x, B_r(0))\]
and consequently
	\[c_\e(p,\mathcal{B}_r(x)) \geq (1+\delta)^{-(n+2)}c_{\e/(1+\delta)}(p,B_r(0)), \quad \mbox{for all} \quad p \in \N.\]

Choose a collection $\{\mathcal{B}_i\}_{i=1}^N$ of pairwise disjoint geodesic balls in $M$ with radius $r_i\leq \bar r$ and such that $\sum_i \vol(\mathcal{B}_i) \geq (1+\delta)^{-1}$. If $B=B_1(0)$ is the unit ball in $\R^{n+1}$ and $B_i$ is for each $i=1,\ldots, N$ an Euclidean ball of radius $r_i$ then we see that
	\[c_\e(p,M) \geq \sum_{i=1}^N c_\e\left( \left\lfloor p \vol(\mathcal{B}_i)\right\rfloor, \mathcal{B}_i \right).\]
Therefore by writing $p_i = \lfloor p \vol{\mathcal{B}_i} \rfloor$ and $\e_i=\e/((1+\delta)|B_i|^{1/(n+1)})$ we get
	\begin{align*}
		p^{-\frac{1}{n+1}}c_\e(p,M) & \geq p^{-\frac{1}{n+1}}c_\e(p_i,\mathcal{B}_i)\\
		& \geq (1+\delta)^{-(n+2)} p^{-\frac{1}{n+1}} c_{\e/(1+\delta)}(p_i,B_i)\\
		& = (1+\delta)^{-(n+2)} p^{-\frac{1}{n+1}} \sum_{i=1}^N |B_i|^{\frac{n}{n+1}} c_{\e_i}(p_i,B)\\
		& \geq (1+\delta)^{-(n+2)} \sum_{i=1}^N |B_i|\left(\frac{p_i}{p|B_i|}\right)^{\frac{1}{n+1}}p_i^{-\frac{1}{n+1}} c_{\e_i}(p_i,B)\\
		& \geq (1+\delta)^{-(2n+3)} \sum_{i=1}^N \vol(\mathcal{B}_i)\left( \frac{\vol(\mathcal{B}_i)}{|B_i|} - \frac{1}{p|B_i|} \right)^{\frac{1}{n+1}}p_i^{-\frac{1}{n+1}}c_{\e_i}(p_i,B).
	\end{align*}
Thus
	\[p^{-\frac{1}{n+1}}\ell_p(M) \geq (1+\delta)^{-(2n+3)}\sum_{i=1}^N \vol(\mathcal{B}_i)\left( \frac{\vol(\mathcal{B}_i)}{|B_i|} - \frac{1}{p|B_i|} \right)^{\frac{1}{n+1}} p_i^{-\frac{1}{n+1}}\ell_{p_i}(B).\]
Finally, the Weyl Law for Euclidean domains implies
	\begin{align*}
		\liminf_p p^{-\frac{1}{n+1}}\ell_p(M) & \geq (1+\delta)^{-(2n+3)} \sum_{i=1}^N \vol(\mathcal{B}_i) (1+\delta)^{-1} \tau(n)|B|^{\frac{n}{n+1}}\\
		& = (1+\delta)^{-2(n+2)} \tau(n) \sum_{i=1}^N \vol(\mathcal{B}_i) \geq (1+\delta)^{-(2n+5)}\tau(n).
	\end{align*}
Since $\delta>0$ is arbitrary, the inequality above concludes the proof of the theorem.
\end{proof}

In order to prove Theorem \ref{thm:weyl_} it remains to show that
	\[\limsup_{p \to +\infty} p^{-\frac{1}{n+1}}\ell_p(M) \leq \tau(n) \vol(M,g)^{\frac{n}{n+1}}.\]
Our proof differs from the one in \cite{LioMarNev} as we don't have the spaces of chains in our disposal to perform the cutting and gluing argument. Nevertheless we will still use the strategy of decomposing $M$ into domains $\{\mathcal{C}_i\}$ with piecewise smooth boundaries which are bi-Lipschitz equivalent to Euclidean domains $C_i$ having also piecewise smooth boundaries, and glue them by small tubes obtaining $\Omega \subset \R^{n+1}$ on which we know that the Weyl law holds. Then we follow the ideas of \cite{Guaraco} and \cite{GasparGuaraco} to modify a given $p$-sweepout $A \in \mathcal{F}_p(\Omega)$ so that it induces a $p$-sweepout of $M$ gluing back the domains $C_i$. More precisely, given $u \in A$ we construct a function $w_{\e} \in H^1(\Omega)$ which vanishes in the boundary of all $C_i$, and this allows us to extend it to $\Omega$ in a way that, roughly,
	\[E_\e(w_\e, \Omega) \leq E_\e(u,\Omega) + 2\sigma \sum_{i=1}^N \mathcal{H}^{n}(\partial C_i)+O(\e),\]
in terms of $\e$. This implies a similar inequality for $c_\e(p,\Omega)$ in terms of $c_\e(p,M)$ and the area of the boundaries $\partial C_i$ similarly to \cite{LioMarNev}. Hence we obtain $\limsup_p \tilde\ell_p(M) \leq \tau(n)|\Omega|^{n/(n+1)}$. The Theorem follows then by noting that we can choose $\Omega$ so that its volume is as close to $\vol(M,g)$ as we want.

Firstly we choose a decomposition of $M$ in a similar manner to \cite[\S 4.2]{LioMarNev}. More precisely given $\eta>0$ there is a collection $\{\mathcal{C}_i\}_{i=1}^N$ of domains in $M$ with piecewise smooth boundary and the following properties. Here we denote by $d_i$ the Euclidean distance function $\mathrm{dist}(x,\partial C_i)$ for $x \in C_i$, for each $i=1,\ldots, N$.
	\begin{enumerate}
		\item Each $\mathcal{C}_i$ is $(1+\eta/2)$-biLipschitz diffeomorphic to a domain $C_i \subset \R^{n+1}$ with piecewise smooth $\partial {C}_i$ endowed with the Euclidean metric.
		\item $\{\overline{\mathcal{C}}_i\}$ covers $M$.
		\item The domains $\mathcal{C}_i$ have mutually disjoint interiors.
		\item Given $\eta_1>0$ there is $s_0>0$ such that, for all $s \in (-s_0,s_0)$, we have $\mathcal{H}^n(\{d_i=s\}\cap C_i)\leq (1+\eta_1) \mathcal{H}^n(\partial C_i)$.
	\end{enumerate}
The existence of such a cover is proved in \cite{LioMarNev}. To see why the last property holds it suffices to see that $\{d_i=s\}$ is contained in a union of spheres of radii $r - s$, for some $r=r(M)>0$. Alternatively we can construct $\{\mathcal{C}_i\}$ using a sufficiently fine triangulation -- or cubulation -- of $M$, as in \cite{GasparGuaraco}. Clearly we may assume that the domains $C_i$ are pairwise disjoint in $\R^{n+1}$. Moreover we can construct an Euclidean domain $\Omega \subset \R^{n+1}$ such that 
	\[\Omega = \bigcup_{i=1}^NC_i \cup \bigcup_{i=1}^{N-1}T_i\]
where each $T_i$ is a tube -- e.g. it is diffeomorphic to $S^{n}\times[0,1]$ -- connecting $C_{i}$ to $C_{i+1}$, such that $\{C_i\}\cup\{T_i\}$ have pairwise disjoint interiors and $|T_i|$ may be chosen as small as we want. In particular we may assume
	\[|\Omega| \leq (1+\eta)^{n+1}\vol(M,g).\]
We can also suppose (by suitably chosing $T_i$ and making $s_0$ smaller, if necessary) that
	\[\mathcal{H}^n(\{\mathrm{dist}(\cdot, \partial T_i)=s\} \cap T_i) \leq (1+\eta_1)\mathcal{H}^n(\partial T_i)\]
for all $|s|\leq s_0$.

We recall now some properties of the $1$-dimensional heteroclinic solution of \eqref{eq:ac}, as presented in \cite[\S 7.3]{Guaraco} Let $\psi$ denote the solution to the IVP
	\[\left\{ \begin{array}{rcl} \psi' & = & \sqrt{2W(\psi)} \\  \psi(0)&=&0 \end{array} \right. .\]
Then $\psi$ solves \eqref{eq:ac} in $\R$ for $\e=1$, and it holds:
	\begin{enumerate}
		\item $|\psi|<1$ and $\psi$ is monotone increasing.
		\item $\psi(s) \to \pm 1$ as $s \to \pm \infty$.
		\item $sW(\psi(s)) \to 0$ as $s \to \pm \infty$.
		\item $\int_{\R} (\psi')^2/2 +W(\psi) = 2\sigma = \int_{-1}^1\sqrt{2W}$.
	\end{enumerate}
Given $\e>0$ we denote also $\psi_\e(s)=\psi(s/\e)$ for $s \in \R$. Clearly $\psi_\e$ solves \eqref{eq:ac} in $\R$. Denote by $d:\bar\Omega \to \R$ the function given by $d(x) = d_i(x)$, if $x \in \bar{\mathcal{C}_i}$, and $d(x)=\mathrm{dist}(x,\partial T_i)$ if $ x \in \bar{T_i}$. Clearly $d$ is a Lipschitz function and it satisfies the Eikonal equation $|\nabla d|=1$ almost everywhere in $\Omega$. Now fix $\delta>0$ and define $v_{\delta,\e}:\Omega \to \R$ by
	\begin{equation} \label{eq:vie}
		v_{\delta,\e}(x) = \left\{ \begin{array}{cl} \psi_\e(d(x)), & \mbox{if}\quad d(x)\leq \delta \\ \psi_\e(\delta), & \mbox{if} \quad d(x) >\delta \end{array} \right. .
	\end{equation}
Again we can verify that $v_{\delta,\e}$ is a Lipschitz function, and moreover
	\[ |\nabla v_{\delta,\e}|_x| = \left\{ \begin{array}{cl} \psi_\e'(d(x)), & \mbox{for a.e. $x \in \{d \leq \delta\} \cap \Omega$} \\ 0, & \mbox{for all $x \in \{d > \delta\} \cap \Omega$} \end{array} \right. . \]
Now given $u:\Omega \to \R$ we define a new function $w_{\e}:\Omega \to \R$ by truncating $u$ by $\pm v_{\delta,\e}$ in $\Omega$, that is
	\[w_{\e}(x) = \max\{-v_{\delta,\e}(x), \min\{ u(x),v_{\delta,\e}(x) \} \} \}, \quad \mbox{for} \quad x \in \Omega.\]
If $u \in H^1(\Omega)$ then we have $w_{\e} \in H^1(\Omega)$ and $|\nabla w_{\e}|=|\nabla u|$ a.e. in $\{|u|\leq v_{\delta,\e}\} \cap \Omega$, whereas $|\nabla w_{\e}| = \psi_\e'(d)$ almost everywhere in $\{|u|>v_{\delta,\e}\}\cap \Omega$. Hence
	\begin{align*}
		E_\e(w_\e, C_i) &= \int_{\{|u|\leq v_{\delta,\e}\}\cap C_i} \left(\frac{\e|\nabla u|^2}{2} + \frac{W(u)}{\e} \right) + \int_{\{ |u|>v_{\delta,\e} \}\cap C_i}\left(\frac{\e\psi_\e'(d_i)^2}{2} + \frac{W(\psi_\e(d_i))}{\e}  \right)\\
		& \leq E_\e(u, C_i) + I_1 + I_2
	\end{align*}
for each $i=1,\ldots, N$, where
	\[I_1  = \int_{\{d_i>\delta\}\cap C_i} \frac{W(\psi_\e(\delta))}{\e} \leq \frac{W(\psi(\delta/\e))}{\e}|C_i|,\]
and
	\begin{align*}
		I_2 & = \int_{\{d_i\leq \delta\}\cap C_i} \left( \frac{\e(\psi_\e'(d_i))^2}{2} + \frac{W(\psi_\e)}{\e} \right)\\
		& = \frac{1}{\e}\int_{-\delta}^{\delta} \left( \frac{\psi'(t/\e)^2}{2} + W(\psi(t/\e)) \right)\mathcal{H}^n(\{d_i=t\}\cap C_i)\,dt\\
		& = \int_{-\delta/\e}^{\delta/\e}\left(\frac{\psi'(s)^2}{2} + W(\psi(s))\right) \mathcal{H}^n(\{d_i=\e s\}\cap C_i)\,ds\\
		&\leq \left( \int_\R (\psi')^2/2 + W(\psi) \right)\left( \sup_{|s|\leq \delta}\mathcal{H}^n(\{d_i=s\})\right) \leq 2\sigma(1+\eta)\mathcal{H}^n(\partial C_i)
	\end{align*}
provided we pick a sufficiently small $\delta$ depending only on $M$ (and the cover $\{\mathcal{C}_i\}$) and $\eta_1 \leq \eta$. Similarly for $i=1,\ldots, N-1$ we have
	\begin{align*}
		E_\e(w_\e,T_i) & = \int_{\{|u|\leq v_{\delta,\e}\cap T_i\}} \left( \frac{\e|\nabla u|^2}{2} + \frac{W(u)}{\e} \right)\\
		& \qquad + \int_{\{|u|>v_{\delta,\e}\cap T_i\}} \left( \frac{\e\psi'(\mathrm{dist}(\cdot,\partial T_i))^2}{2} + \frac{W(\psi_\e(\mathrm{dist}(\cdot,\partial T_i)))}{\e} \right)\\
		& \leq E_\e(u,T_i) + \frac{W(\psi(\delta/\e))}{\e}|T_i| + 2\sigma(1+\eta)\mathcal{H}^n(\partial T_i)
	\end{align*}
for sufficiently small $\delta>0$. Therefore
	\begin{align*}
		E_\e(w_\e,\Omega) & = \sum_{i=1}^N E_\e(w_\e,C_i) + \sum_{i=1}^{N-1}E_\e(w_\e,T_u)\\
		& \leq E_\e(u,\Omega) + \frac{W(\psi(\delta/\e))}{\e}|\Omega| +2\sigma(1+\eta)\beta(\Omega),
	\end{align*}
where $\beta(\Omega) = \sum_{i=1}^N \mathcal{H}^n(\partial C_i) + \sum_{i=1}^{N-1}\mathcal{H}^n(\partial T_i)$. On the other hand, since $w_\e$ vanishes on $\partial C_i$ for all $i$, we may use the $(1+\eta/2)$-bilipschitz equivalence $F_i:\mathcal{C}_i \to C_i$ to define $U_\e:M\to \R$ by $U_\e(x) = (w_\e\circ F_i)(x)$ for $x \in \bar{\mathcal{C}_i}$, so that $U_\e \in H^1(M)$ and by Lemma \ref{lema:var}
	\[E_{\e/(1+\eta/2)}(U_\e,\mathcal{C}_i) \leq (1+\eta/2)^{n+2} E_\e(w_\e,C_i).\]
Thus
	\[E_{\e/(1+\eta/2)}(w_\e, M) \leq (1+\eta/2)^{n+2}\left(E_\e(u,\Omega) + \frac{W(\psi(\delta/\e))}{\e}|\Omega| + 2\sigma(1+\eta)\beta(\Omega)\right).\]
If we prove that $u \in H^1(\Omega) \mapsto U_\e \in H^1(M)$ defines a continuous odd map then the monotonicity of the $\Z/2\Z$ index and the inequality above give us
	\[ c_{\e/(1+\eta/2)}(p,M) \leq (1+\eta/2)^{n+2}\left( c_\e(p,\Omega) + \frac{W(\psi(\delta/\e))}{\e}|\Omega| + 2\sigma(1+\eta)\beta(\Omega)\right) \]
for all $p \in \N$. Consequently
	\[ \ell_p(M) \leq (1+\eta/2)^{n+2}\left( \ell_p(\Omega) + 2\sigma(1+\eta)\beta(\Omega)\right)\]
and 
	\begin{align*}
		\limsup_p p^{-\frac{1}{n+1}}\ell_p(M) & \leq (1+\eta/2)^{n+2}\limsup_p p^{-\frac{1}{n+1}}\ell_p(\Omega)& \\
		& = (1+\eta/2)^{n+2}\tau(n)|\Omega|^{\frac{n}{n+1}} \\
		&\leq (1+\eta)^{2n+3}\tau(n)\vol(M,g)^{\frac{n}{n+1}}.
	\end{align*}
The continuity of the truncation and gluing construction used above is a consequence of the following lemma.

\begin{lemma}
Fix $\e,\delta>0$ and consider the decomposition $\{\mathcal{C}_i\}$ and the bilipschitz diffeomorphisms $F_i:\mathcal{C}_i \to C_i \subset \R^{n+1}$ described above, and the function $v_{\delta,\e}$ defined in \eqref{eq:vie}. The map $\Psi:W^{1,2}(\Omega) \to W^{1,2}(M)$,
	\[(\Psi u)(x) := \max\{ -(v_{\delta,\e}\circ F_i)(x), \min\{ (v_{\delta,\e}\circ F_i)(x),(u\circ F_i)(x) \}  \}, \quad \mbox{for} \quad x \in \mathcal{C}_i\]
is odd and continuous.
\end{lemma}

\begin{proof}
By the continuity of the maximum and minimum functions in $H^1$ we see that
	\[(\bar \Psi u)(x) := \max\{-v_{\delta,\e}(x),\min\{v_{\delta,\e}(x),u(x)\}\}, \quad x \in {\textstyle\bigcup_i C_i}\]
defines a continuous map $\bar\Psi:H^1(\Omega) \to H^1_0(\bigcup_i C_i)$.

On the other hand, since each $F_i$ is a bilipschitz diffeomorphism, the map
	\[\Phi_i:H^{1}_0(C_i) \ni u \mapsto u \circ F_i \in H^{1}_0(\mathcal{C}_i)\]
is also continuous. By extending $\Phi_i u$ to $0$ outside $\mathcal{C}_i$ we get a continuous map into $H^{1}(M)$. Using also that $W^{1,2}_0(\bigcup_i C_i) = \oplus_i W^{1,2}_0(C_i)$ we can put
	\[\Phi:{\textstyle W^{1,2}_0\left(\bigcup_i C_i\right) \to W^{1,2}(M),\quad \Phi u =  \sum_{i=1}^N} \Phi_i(u|_{C_i}).\]
Then $\Phi$ is also continuous and the claimed result follows by noting that $\Psi = \Phi \circ \bar\Psi $.
\end{proof}

\section{Density of Limit interfaces}\label{sec:density}

The proofs of the main theorems in \cite{IrieMarNev} and \cite{MarNevSon} rely on the Weyl law for the volume spectrum together with some perturbation arguments. In our context, we may replace the former by our main theorem, the Weyl law for the phase transition spectrum. While the perturbation arguments remain roughly unchanged, some adaptations are needed. Before talking about limit interfaces we will indicate the changes needed in \cite{IrieMarNev} and \cite{MarNevSon} in order to obtain phase transition based proofs of:

\begin{teo} \label{thm:den}(from \cite{IrieMarNev})
Let $M^{n+1}$ be a closed manifold of dimension $3\leq n+1 \leq 7$. For a $C^\infty$-generic Riemannian metric $g$ on $M$, the union of all closed, smooth, embedded minimal hypersurfaces in $(M,g)$ is dense. 
\end{teo}

\begin{teo} \label{thm:equi}(from \cite{MarNevSon})
Let $M^{n+1}$ be a closed manifold of dimension $n+1$, with $3 \leq n + 1 \leq 7$. Then for a $C^\infty$-generic Riemannian metric $g$ on $M$, there exists a sequence $\{\Sigma_j\}_{j\in \N}$ of closed, smooth, embedded, connected minimal hypersurfaces that is equidistributed in $M$: for any $f \in C^\infty(M)$ one has
$$\lim_{q\to\infty}\frac{1}{\sum_{j=1}^q \operatorname{vol_g}(\Sigma_j)} \sum_{j=1}^{q} \int_{\Sigma_j}f \ d\Sigma_j = \frac{1}{\operatorname{vol}_g M} \int_M f dM.$$
Even more, for any symmetric $(0, 2)$-tensor $h$ on $M$, one has
$$\lim_{q\to\infty}\frac{1}{\sum_{j=1}^q \operatorname{vol_g}(\Sigma_j)} \sum_{j=1}^{q} \int_{\Sigma_j}\operatorname{Tr}_{\Sigma_j} (h) \ d\Sigma_j = \frac{1}{\operatorname{vol}_g M} \int_M \frac{n \operatorname{Tr}_M h}{n+1} dM.$$

\end{teo}

Moreover, combining the argument of \cite{IrieMarNev} with the construction of Pacard-Ritor\'e \cite{PacardRitore} and with the multiplicity one of the interfaces in 3-dimensional manifolds of Chodosh-Mantoulidis \cite{ChodoshMantoulidis}, we obtain

\begin{teo} \label{thm:denallencahn}
Let $M^{n+1}$ be a closed manifold of dimension $(n+1)$, such that 
\begin{enumerate}
\item $n+1=3$, or 
\item $4\leq n+1 \leq 7$ and $M$ contains only separating minimal hypersurfaces e.g. $H_{n}(M,\Z_2)=0.$
\end{enumerate} Then, for a $C^\infty$-generic Riemannian metric $g$ on $M$, the union of all closed, smooth, embedded separating limit interfaces in $(M,g)$ is dense. 
\end{teo}

First, we need to prove that the volume spectrum depends continuously on the metric:

\begin{lemma}\label{lem:cont}
The $p$-th liminf (resp. limsup) min-max value for the energy $\underline\ell_p(M,g)$ (resp. $\overline\ell_p(M,g))$ depends continuously on the metric $g$, with respect to the $C^0$ topology.
\end{lemma}

\begin{proof}
Assume $g_i$ is a sequence of smooth Riemannian metrics converging in the $C^0$ topology to $g$. For all $i$ write
	\[\lambda_i = \max\left\{\left( \sup_{v \neq 0} \frac{g_i(v,v)}{g(v,v)} \right)^{1/2},\left(\sup_{v \neq 0} \frac{g(v,v)}{g_i(v,v)} \right)^{1/2}\right\}\]
so that $\lambda_i^{-2}g\leq g_i \leq \lambda_i^2 g$ and $\lambda_i \to 1$, by the convergence $g_i \to g_0$ in $C^0$. Note that this implies that the $H^1$ norm induced by each $g_i$ is equivalent to the one induced by $g$. Given $\e>0$ choose a compact and symmetric subset $A \subset H^{1}(M)$ such that 
	\[\sup_{u \in A} E_\e(u,g) \leq c_\e(p,g) + \e.\]
where $E_\e(\cdot, g)$ is the Allen-Cahn energy calculated with respect to the metric $g$. Proceeding as in the proof of Lemma \ref{lema:var}, we get
	\[E_{\e/\lambda_i}(u,g_i) \leq \lambda_i^{n+2} E_\e(u,g)\]
for all $H^1$ functions $u$ on $M$. Moreover $A \in \mathcal{F}_p(M,g_i)$ (as a subset of $H^1(M,g_i)$) and thus
	\[	c_{\e/\lambda_i}(p,g_i) \leq \sup_{u \in A} E_{\e/\lambda_i }(u,g_i) \leq \lambda_i^{n+2} \sup_{u \in A} E_\e(u,g) \leq \lambda_i^{n+2}(c_\e(p,g)+\e).\]
Hence
	\[\limsup_i \underline\ell_p(M,g_i) \leq \underline\ell_p(M,g)\]
and
	\[\limsup_i \overline\ell_p(M,g) \leq \overline\ell_p(M,g).\]
Similarly, we may prove that $\liminf_i \underline\ell_p(M,g_i) \geq \underline \ell_p(M,g)$ and $\liminf_i\overline\ell_p(M,g_i) \geq \overline\ell_p(M,g)$.
\end{proof}

In fact, we can proof the following equivalent to Lemma 1 from \cite{MarNevSon}.

\begin{cor}
\label{Lipschitz}
Let $\tilde g$ be a $C^2$ Riemannian metric on $M$, and let $C_1<C_2$ be positive constants. Then there is $K=K(\tilde g, C_1, C_2)>0$ such that $$|p^\frac{1}{n+1} \underline \ell_p(M,g')-p^\frac{1}{n+1} \underline \ell_p(M,g)| \leq K \cdot |g-g'|_{\tilde g},$$ for any $g,g'\in \{h\in \Gamma_2; C_1 \tilde g \leq h \leq C_2 \tilde g\}$ and any $p\in \N$.
\end{cor}

\begin{proof}
The proof is identical to Lemma 1 from \cite{MarNevSon}, it uses the Gromov-Guth sublinear bounds, which in the phase transitions context were proved by the authors on \cite{GasparGuaraco}; and the proof of the equivalent of Lemma \ref{lem:cont}, which was proved by Irie-Marques-Neves in \cite{IrieMarNev}. By substituting this two results in the proof we obtain the Lipschitz continuity. 
\end{proof}

The next result tells us that $\overline\ell_p(M)$ and $\underline\ell_p(M)$ are achieved by limit interfaces with Morse index at most $p$, and its proof follows directly from \cite{GasparGuaraco} and \cite{Gaspar}.

\begin{prop} \label{prop:ach}
Suppose $3 \leq n+1 \leq 7$. Then for each $p \in \N$ there exist finite disjoint collections $\{\Gamma_1,\ldots, \Gamma_U\}$ and $\{\Sigma_1,\ldots, \Sigma_L\}$ of closed, smooth, embedded minimal hypersurfaces in $M$, and positive integers $\{\overline m_1,\ldots, \overline m_U\}$ and $\{\underline m_1,\ldots, \underline m_L\}$ such that
	\[\overline\ell_p(M,g) = \sum_{j=1}^U \overline m_j \vol_g(\Gamma_j), \quad \underline\ell_p(M,g) = \sum_{j=1}^L \underline m_j \vol_g(\Sigma_j)\]
and
	\[\sum_{j=1}^U \operatorname{Ind}(\Gamma_j)\leq p, \quad \sum_{j=1}^L \operatorname{Ind}(\Sigma_j) \leq p.\]
Furthermore $\sum_{j=1}^U \overline m_j \Gamma_j$ and $\sum_{j=1}^L \underline m_j \Sigma_j$ are limit interfaces, that is the limits of the varifolds associated to sequences of solutions to the Allen-Cahn equation with Morse index at most $p$ and parameter $\e$ converging to $0$.
\end{prop}

\begin{obs}
Different from Proposition 2.2 in \cite{IrieMarNev} we do not need to use Sharp's Compactness Theorem \cite{Sharp} in the proof of Proposition \ref{prop:ach}. In Marques-Neves' setting this happens because the $p$-widths are defined using cohomological classes of maps into the space of $n$-cycles modulo $\Z_2$ while Almgren-Pitts Regularity theory works with homotopy classes. In our case, the phase transition spectrum and also the existence theorems of \cite{GasparGuaraco} may be described in terms of cohomological families while the convergence to a smooth limit interface is independent of these constructions.
\end{obs}

Finally we need a version of Proposition 2.3 in \cite{IrieMarNev} which preserves separating limit interfaces. This is the content of the next result.

\begin{prop} \label{prop:def}
Let $\Gamma$ be a closed, smooth, embedded and separating minimal hypersurface in $(M^{n+1},g)$. Then, there exists a sequence of metrics $g_i$ converging to $g$ in the smooth topology such that $\Gamma$ is a nondegenerate limit interface for each $(M^{n+1},g_i)$.
\end{prop}

\begin{proof}
By Proposition 2.3 \cite{IrieMarNev} we know there exists a sequence of metrics $g_i$ such that $\Gamma$ is a nondegerate minimal hypersurface on every $(M^{n+1},g_i)$. Since $\Gamma$ is also separating, Theorem 4.1 of \cite{PacardRitore} gives, for every $g_i$, the existence of $\e_0$ such that for each $\e\in(0,\e_0)$ there is a solution to $\Delta_{g_i}u_\e-W'(u_\e)$, having $\Gamma$ as a limit interface, as $\e\to 0$.
\end{proof}

\vspace{20pt}

{\bf Density and Equidistribution of Minimal Hypersurfaces.} 

\begin{proof}[Proof of Theorem \ref{thm:den}]
Similarly to \cite{IrieMarNev}, we will show that given an open set $U \subset M$ the space $\mathcal{M}(U)$ of all smooth metrics on $M$ for which there exists an embedded nondegenerate minimal hypersurface intersecting $U$ is open and dense, with respect to the $C^\infty$ topology.

Let $g \in \mathcal{M}(U)$ as above. The openness of $\mathcal{M}(U)$ follows from the Inverse Theorem Function Theorem applied to the Jacobi operator of $\Gamma$, or from White Structure Theorem \cite{WhiteBumpy} as argued in Proposition 3.1 of \cite{IrieMarNev}. 

To see that $\mathcal{M}(U)$ is dense one can proceed as in the proof of Proposition 3.1 \cite{IrieMarNev} substituting the continuity of the area functional spectrum, i.e. Lemma 2.1\cite{IrieMarNev} by Lemma \ref{lem:cont}, Propositon 2.2 \cite{IrieMarNev} by Propostion \ref{prop:ach} and restricting the set of all possible areas of min-max stationary varifolds to the possible areas of limit interfaces, all of on a fixed bumpy metric. The Weyl law for the volume spectrum Theorem \cite{LioMarNev} is then replaced by the Weyl law for the phase transition spectrum, Theorem \ref{thm:weyl_} and the same contradiction argument shows that we can find an arbitrarily small deformation $g'$ of a bumpy metric in $U$ so that some limit interface intersects $U$. Then, the deformation of Proposition 2.3 \cite{IrieMarNev} concludes the density of $\mathcal{M}(U)$.
\end{proof}

\begin{proof}[Proof of Theorem \ref{thm:equi}]
We indicate how each one of the lemmas in \cite{MarNevSon} is affected when one wants to use the Weyl law for the phase transition spectrum. Lemma 1 \cite{MarNevSon} is based on the Guth-Gromov Bounds from Marques-Neves \cite{MarquesNevesInfinite} and Lemma 2.1 \cite{IrieMarNev}, which should then be replaced by Theorem 3.2 \cite{GasparGuaraco} and Lemma \ref{lem:cont}, respectively. Lemma 2 \cite{MarNevSon} uses only Proposition 2.2 \cite{IrieMarNev}, which can be replaced by our Proposition \ref{prop:ach}. Lemma 3 and 4 may then be used verbatim. The main result in \cite{MarNevSon} is then consequence of Lemmas 1, 2, 3 an 4, \cite{MarNevSon}. 
\end{proof}

\

{\bf Density of Limit Interfaces.} 
\begin{proof}[Proof of Theorem \ref{thm:denallencahn}]

Differently from above, let $\mathcal{M}_2(U)$ be all smooth metrics on $M$ for which there exists a closed, two-sided, separating, smooth and embedded nondegenerate minimal hypersurface (not necessarily connected) intersecting $U$. We show that this set is open and dense, with respect to the $C^\infty$ topology.

The openness of $\mathcal{M}_2(U)$ follows from a similar argument as in Theorem \ref{thm:den} since the hypersurface given by the application of the Inverse Function Theorem is presented as a normal graph. Therefore, it is also two-sided, separating and non-degenerated. Then Theorem 4.1 of \cite{PacardRitore} implies that it is also a limit interface.

To see that $\mathcal{M}_2(U)$ is dense, the idea is again to proceed as in the proof of Proposition 3.1 \cite{IrieMarNev}, nonetheless, it is necessary to know that the limit interfaces given by the Allen-Cahn min-max, Proposition \ref{prop:ach}, are also separating. This is directly the case if one assumes (1) $n+1=3$ and $\operatorname{Ric}(M)>0$, as a consequence of the multiplicity one property recently shown by Chodosh-Mantoulidis \cite{ChodoshMantoulidis} or if (2) $3\leq n+1\leq 7$ and $M$ contains only separating minimal hypersurfaces as is the case when $H_n(M,\Z_2)=0$. The proof is then the same as in Theorem \ref{thm:den} except in the last step, where we can replace Proposition 2.3 \cite{IrieMarNev} by Proposition \ref{prop:def} in the deformation argument on $U$.

In order to show the result for $n+1=3$ in a generic set of metrics, the fact that $\mathcal{M}_2(U)$ is dense must be argued differently. The reason for this is that the deformed metric on $U$ is not bumpy in general. Therefore, Chodosh-Mantoulidis Sheet Convergence Theorem \cite{ChodoshMantoulidis} does not rule out the possibility of having only unstable one-sided interfaces, or non-separating interfaces with multiplicity,  intersecting the set $U$. 

To fix this problem we consider the following local version of B. White's Bumpy Metrics Theorem (which was announced by White in \cite{WhiteBumpy}). Roughly speaking, the result will be used to produce a deformation of $U$ that preserves bumpiness allowing us to apply the Sheet Convergence Theorem \cite{ChodoshMantoulidis} to conclude that the interfaces are separating minimal hypersurfaces. We present the statements now but we postpone its proofs to the end of this section.

\begin{teo}[Local Bumpy Metrics Theorem]\label{bumpymetrics} Let $g$ be a $C^q$ Riemannian metric on $M$ and $U\subset M$ an open set. Define the Banach space $$\mathcal{S}(U)=\{\gamma : \gamma= 0 \textit{ on } M\setminus U \},$$ where $\gamma$ varies on the space of $C^q$ sections of symmetric bilinear forms on $M$, and its open subset $$\Gamma_g(U) = \{\gamma \in \mathcal{S}(U) : g + \gamma \text{ is a metric on } M \}.$$ Then, the set of $\gamma \in \Gamma_g(U)$ such that any component of a $(g + \gamma)$-minimal hypersurface intersecting $U$ is non degenerate in $(M,g+\gamma)$, is a generic subset of $\Gamma_g (U)$.
\end{teo}

This is a consequence of following theorem, where $[w]$ represents the equivalence class of a $C^{j,\alpha}$ embedding $w : \Sigma \to M$, modulo diffeomorphisms of $\Sigma$, with $q\geq j+1\geq 3.$

\begin{teo}[Local Manifold Structure Theorem]\label{StructureTheorem}  Following the same notation as in the last Theorem, let $\Sigma$ be a smooth $n$-dimensional closed Riemannian manifold and $\Gamma$ an open subset of $\Gamma_g(U)$. Denote by $\mathcal{M}_g(\Sigma, U)$ the set of ordered pairs $(\gamma, [w])$ where $\gamma \in \Gamma$,  $w\in C^{j,\alpha}(\Sigma,M)$ is a $(g+\gamma)$-minimal embedding and $w(\Sigma)\cap U\neq \emptyset$. 

Then $\mathcal{M}_g(\Sigma, U)$ is a separable $C^{q-j}$ Banach manifold modelled on $\Gamma$, and the map $$\Pi: \mathcal{M}_g(\Sigma, U) \to \Gamma$$
$$\Pi(\gamma,[w])=\gamma $$
is a $C^{q-j}$ Fredholm map with Fredholm index 0. Moreover, the kernel of $D\Pi(\gamma,[w])$ has dimension equal to the nullity of $[w]$ with respect to $g+\gamma$, in particular $(\gamma,[w])$ is a critical point for $\Pi$ if and only if the embedding $w$ admits non trivial Jacobi fields with respect to $g+\gamma$.
\end{teo}

\

\noindent\textbf{Remark.}\textit{
We emphasize that $\Gamma$ is not a set of metrics, but an open set of sections of symmetric bilinear forms $\gamma$, with $\operatorname{supp}\gamma \subset \overline{U}$ such that $g+\gamma$ is a metric on $M$.}

\

These results allows us argue following ideas from Lemma 2 of \cite{MarNevSon}. Let $\tilde g$ be a  metric on $M$ and $\mathcal{V}$ an open set of metrics containing $\tilde g$. Fix $g_0 \in \mathcal{V}$ a bumpy metric on $M$. Remember that, as mentioned before, there exists a generic set of such metrics. 

Define $\mathcal{M}_{g_0}(U)=\cup_i \mathcal{M}_{g_0}(\Sigma_i,U)$, where $\{\Sigma_i\}_i$  enumerates all the diffeomorphism types of closed manifolds of dimension $n$. By Theorem \ref{StructureTheorem}, $\mathcal{M}_{g_0}(U)$ is a separable $C^{q-2}$ Banach manifold and $\Pi: \mathcal{M}_{g_0}(U) \to \Gamma_{g_0}(U)$ is a Fredholm map of index 0.

Let $g(t)=g_0+\gamma(t)$ be a \textit{smooth} deformation of the metric $g_0=g(0)$ on the set $U$, such that $g(t) \in \mathcal{V}$ and $\vol(M,g(t)) > \vol(M,g_0)$ for all non-zero $t\in I=[0,1]$, as constructed in Proposition 3.1 of \cite{IrieMarNev}.  

Since the metric $g_0$ is bumpy, it follows from Theorem \ref{StructureTheorem} that $D\Pi$ has no kernel on the points of $\mathcal{M}_{g_0}(U)$ with first coordinate $\gamma=0$, i.e. before deforming the metric $g_0$. Since $\Pi$ is a Fredholm map of index 0, this implies that the derivative of $\Pi$ is onto. In particular, the maps $\gamma$ and $\Pi$ are transversal at $\gamma(0)=0$. Moreover, by Smale's Transversality Theorem (Theorem 3.1 \cite{Smale}) there exists $\gamma'$, a perturbation of $\gamma$ arbitrarily small on the $C^\infty$ topology, such that $\gamma'$ and $\Pi$ are transversal maps,  $\gamma'(0)=\gamma(0)=0$ and $J =\Pi^{-1}(\gamma(I))$ is a smooth embedded curve on $M_{g_0}(U)$. In particular, we can assume that $g'(t)=g_0+\gamma'(t) \in \mathcal{V}$, for all $t\in I$ and that $\vol(M,g'(1)) > \vol(M,g_0)$. Note also that $\gamma(t) \in \Gamma_{g_0}(U)$ by construction.

Let $\mathcal{A}$ be the set of regular values of the map $\pi: J \to I$, given by $\pi=(\gamma')^{-1} \circ \Pi |_{J}$. $\mathcal{A}$ is a set of full measure by Sard's Theorem. Therefore, reparametrizing if necessary, we can assume that $1$ is a regular value for $\pi$. Since $\vol(M,g'(1)) > \vol(M,g_0)$ by the Weyl Law for the Phase Transition Spectrum, Theorem \ref{thm:weyl_}, there must be a $p \in \N$, for which $\ell_p(M,g'(1)) > \ell_p(M,g_0)$. Since the function $f_p: t \in I \mapsto \ell_p(M,g'(t))$ is Lipschitz by Corollary \ref{Lipschitz}, it follows that $f_p(I\setminus \mathcal{A})$ has null measure. On the other hand, by continuity $[\ell_p(M,g(0)),\ell_p(M,g'(1))] \subset f_p(I)$. This implies that $f_p(\mathcal{A})$ has positive measure.

We will now reach a contradiction. Assume that for all $t\in \mathcal{A}$, limit interfaces never intersect $U$. On one hand, we have that $g'(t)|_{M\setminus U}=g_0|_{M\setminus U}$ is a fixed bumpy metric, for all $t \in I$.  Then, $f_p(\mathcal{A})$ must be contained in the set of possible values for $\ell_k(M,g_0)$, $k\in\N$. On the other hand, this is a countable set for bumpy metrics, by Sharp Compactness Theorem \cite{Sharp} and Proposition \ref{prop:ach}.

Therefore, there exists $t_0\in \mathcal{A}$ such that $\ell_p(M,g'(t_0))$ is attained by a limit interface intersecting $U$. Notice that the metric $g'(t_0) \in \mathcal{V}$ is bumpy, since for minimal hypersurfaces that do not intersect $U$ the relevant ambient metric coincides with $g_0$, which we chose to be bumpy. Additionally, components intersecting $U$ are also nondegenerate since $t_0$ is a regular value of $\pi$ and therefore $\gamma(t_0)$ is a regular value of $\Pi$. By Chodosh-Mantoulidis Sheet Convergence Theorem it follows that the phase transition spectrum is attained by limit interfaces with multiplicity one on this metric and therefore separating.

\end{proof}

We now present the proof of the Local version of Structure Theorem by B. White.

\begin{proof}
We adapt the proof of B. White's Manifold Structure Theorem (Theorem 2.1, \cite{White1}) to our setting.  As in Theorem 2.1 \cite{White1}, we can parametrize a small open neighborhood of a given $[w_0]$ by an equivalence class of sections $u:\Sigma \to V$, modulo diffeomorphisms of $\Sigma$, where $V$ is a normal vector bundle over $w_0(\Sigma)$ with respect to a fixed smooth background metric on $M$.  

The main tools in proving White's Manifold Structure Theorem are Theorem 1.1 \cite{White1} and Theorem 1.2 \cite{White1}. We claim that in our case all the hypothesis of such theorems (even Hypothesis (C), which is discussed below) are satisfied by replacing the functionals $A_{\gamma}$ and $H(\gamma,\cdot)$ by $A_{g+ \gamma}$ and $H(g+\gamma,\cdot)$, respectively. More precisely, let $G$ be the Banach space of $C^q$ functions $f$ that assign to each $x \in M$, $v\in V_x$, and linear map $L: {T}_x M \to V_x$ a real number $f(x,v,L)$ in such a way that $D_3 f$ is also $C^q$. In order to apply Theorem 1.1 \cite{White1} to our case we need for $\gamma \mapsto A_{g+\gamma}$ to be a smooth map from $\Gamma$ to $G$. Similarly, to apply Theorem 1.2 \cite{White1} we need for $$\gamma \times u \mapsto A(g+\gamma,u)=\int_M A_{g+\gamma} (x,u(x),\nabla u(x)) dx$$ to be $C^2$ and for $\gamma\times u \to H(g+\gamma,u)$ to be $C^q$. This follows since the function $\gamma \mapsto g + \gamma$ is just a translation by a constant fixed metric, so all the differentiability properties on the parameter $\gamma$ are preserved for the translated maps.

To see that Hypothesis (C) also holds in our case, namely, that given $(\gamma_0,[u_0]) \in \mathcal{M}_g(\Sigma, U)$ and $\kappa \in \ker D_2 H(\gamma_0,u_0)$, we can find a family $\gamma_s \in \Gamma$ so that $$\bigg(\frac{\partial^2}{\partial s \partial t}\bigg)_{(s=t=0)} \int_M A_{g+\gamma_s}(x,u_0+t\kappa,\nabla(u_0+t \kappa)))dx \neq 0,$$ we argue that the conformal family of metrics $$g_s(z)=(1+sf(z))(g(z)+\gamma_0(z))$$ constructed by B. White can be chosen so that the function $f$ has support on the open set $U$. In this way $g_s = g + \gamma_s$, with $\gamma_s(z)=\gamma_0(z)+sf(z)g(z) \in \Gamma$, for small values of $s$. Following the computation in \cite{White1} one would have
\begin{align*}
\bigg(\frac{\partial^2}{\partial s \partial t}\bigg)_{(s=t=0)} &\int_M A_{g+\gamma_s}(x,u_0+t\kappa,\nabla(u_0+t \kappa)))dx \\
=&\int_M \frac{n}{2}[\nabla f(E(x,u_0))\cdot D_2E(x,u_0(x))\kappa(x)]A_{g+\gamma_0}(x,u_0(x),\nabla u_0) dx
\end{align*}
Indeed, notice that not only $A_{g+\gamma_0}>0$, as in the original proof, but also the Jacobi vector field $\kappa$ cannot be identically zero over the non-empty open set $w_0^{-1}(U)\subset \Sigma$ by the (weak) principle of unique continuation for Schr\"odinger operators. In this way we can find a $C^q$ function $f$, with $\operatorname{supp}f \subset U$, such that the second integral is not zero.

The last steps of the proof of the Structure Theorem of White adapt to our situation verbatim.

\end{proof}

Finally we prove, 

\begin{proof}[Proof of Theorem \ref{bumpymetrics}]
Consider a sequence $\{\Sigma^n_i\}_{i,n}$, that enumerates all the diffeomorphism types of closed manifolds of dimension $n$. Let $\mathcal{M}(\Sigma^n_i,U)$ be the set  given by Theorem \ref{StructureTheorem} for $\Sigma=\Sigma^n_i$.  Since $\mathcal{M}(\Sigma^n_i,U)$ is separable and the projection $\Pi$ is proper, the regular values of $\Pi$ are generic by the Sard-Smale Theorem. Since $\{\Sigma^n_i \}_{i,n}$ is countable we conclude the proof. 
\end{proof}

\bibliography{main}
\bibliographystyle{siam}
	
\end{document}